\documentclass{amsart}
\usepackage[utf8]{inputenc}
\usepackage{lipsum}
\usepackage{amsmath}
\usepackage{amssymb}
\usepackage{amsthm}
\usepackage{fancyhdr}
\usepackage{graphicx}
\usepackage{verbatim}
\usepackage{latexsym}
\usepackage[color = lightgray]{todonotes}
\usepackage{mathrsfs}
\usepackage{tikz}
\usepackage{tikz-cd}
\usepackage{float}
\usepackage{enumerate}
\usepackage[all]{xy}
\usepackage{cite}
\newcommand{\Pref}[1]{Proposition~\ref{#1}}
\newcommand{\Lref}[1]{Lemma~\ref{#1}}
\newcommand{\Tref}[1]{Theorem~\ref{#1}}
\newcommand{\Eref}[1]{Example~\ref{#1}}
\newcommand{\Rref}[1]{Remark~\ref{#1}}

\tikzset{commutative diagrams/.cd,
pushout/.style={start anchor=center, end anchor=center, draw=none}
}

\renewcommand{\todo}[1]{}
\newcommand{\Q}{\mathbb{Q}}
\newcommand{\Z}{\mathbb{Z}}

\newcommand{\A}{\mathscr{A}}
\newcommand{\B}{\mathscr{B}}
\newcommand{\C}{\mathscr{C}}
\newcommand{\D}{\mathscr{D}}
\newcommand{\X}{\mathscr{X}}
\newcommand{\Y}{\mathscr{Y}}

\newtheorem{bigthm}{Theorem}
\newtheorem{thm}{Theorem}[section]
\newtheorem{lemma}[thm]{Lemma}
\newtheorem{cor}[thm]{Corollary}
\newtheorem{prop}[thm]{Proposition}

\theoremstyle{definition}
\newtheorem{defn}[thm]{Definition}
\newtheorem{ex}[thm]{Example}

\theoremstyle{remark}

\newtheorem{remark}[thm]{Remark}
\newcommand{\op}{^{\operatorname{op}}}
\newcommand{\pprime}{^{\prime\prime}}
\DeclareMathOperator{\Add}{Add}
\DeclareMathOperator{\Rep}{Rep}
\DeclareMathOperator{\Filt}{Filt}
\DeclareMathOperator{\sFilt}{sFilt}
\DeclareMathOperator{\Ext}{Ext}
\DeclareMathOperator{\ext}{ext}
\DeclareMathOperator{\add}{add}
\DeclareMathOperator{\Hom}{Hom}
\DeclareMathOperator{\Ch}{Ch}
\DeclareMathOperator{\Ab}{Ab}
\DeclareMathOperator{\Rmod}{R-Mod}
\DeclareMathOperator{\Inj}{Inj}
\DeclareMathOperator{\Flat}{Flat}
\DeclareMathOperator{\Proj}{Proj}
\DeclareMathOperator{\proj}{proj}
\DeclareMathOperator{\dlim}{\varinjlim}
\DeclareMathOperator{\coker}{coker}
\newcommand{\subsetref}[1]{\overset{\textnormal{#1}}{\subseteq}}
\DeclareMathOperator{\GInj}{GInj}
\DeclareMathOperator{\GFlat}{GFlat}
\DeclareMathOperator{\wGFlat}{wGFlat}
\DeclareMathOperator{\GProj}{GProj}
\DeclareMathOperator{\Gproj}{Gproj}

\title{Direct Limit closure of induced Quiver Representations}
\author{Rune Harder Bak}
\address{Department of Mathematical Sciences, University of Copenhagen, Universitetsparken~5, 2100 Copenhagen {\O}, Denmark} 
\email{bak@math.ku.dk}

\begin{document}
\begin{abstract}
  In 2004 and 2005 Enochs et al. characterized the flat and projective quiver-representations of left rooted quivers. The proofs can be understood as filtering the classes $\Phi(\operatorname{Add}\mathscr X)$ and $\Phi(\varinjlim\mathscr X)$ when $\mathscr X$ is the finitely generated projective modules over a ring. In this paper we generalize the above and show that $\Phi(\mathscr X)$ can always be filtered for any class $\mathscr X$ in any AB5-abelian category. With an emphasis on $\Phi(\varinjlim\mathscr X)$ we investigate the Gorenstein homological situation. Using an abstract version of Pontryagin duals in abelian categories we give a more general characterization of the flat representations and end up by describing the Gorenstein flat quiver representations over right coherent rings.
\end{abstract}
\maketitle
\section*{Introduction}
Let $Q$ be a quiver (i.e. a directed graph) and consider for a class $\X$ of objects in an abelian category $\A$ the class $\Phi(\X)\subseteq\Rep(Q,\A)$ of quiver representations. This is the class containing all representations, $F$, s.t. the canonical map $\bigoplus_{w\to v}F(w)\to F(v)$ is monic and has cokernel in $\X$ for all verteces $v$ -  the sum being over all arrows to $v.$
When $Q$ is left-rooted (i.e $Q$ has
no infinite sequence of composable arrows of the form $\cdots \to \bullet \to \bullet \to \bullet$)
it was observed by Enochs, Oyonarte and Torrecillas in \cite{enochsflat} and Enochs and Estrada in \cite{enochsproj}
that when $\A$ is the category of modules over a ring,
\begin{align}
\Phi(\Proj(\A))&=\Proj(\Rep(Q,\A),\text{ and } \label{intro:proj} \\
\Phi(\Flat(\A))&=\Flat(\Rep(Q,\A). \label{intro:flat} 
\end{align}
Here the flat objects are precisely the direct limit closure of the finitely generated projective objects. This was done by showing, that if $\X$ is the finitely generated projective modules over a ring we can filter the classes
$\Phi(\Add\X)$ and $\Phi(\dlim\X)$ by sums of objects of the form $f_*(\X)$
where $f_v\colon\A\to\Rep(Q,\A)$ is the left-adjoint of the evaluation functor $e_v\colon\Rep(Q,\A)\to\A$ at the vertex $v.$
They show
\begin{align}
\Phi(\Add\X)&=\Add f_*(\X) \label{intro:add}\\
\Phi(\dlim\X)&=\dlim\add f_*(\X). \label{intro:dlim}
\end{align}
In 2014 Holm and J\o rgensen \cite{holm} generalized  (\ref{intro:proj}) to abelian categories with enough projective objects, and combining \cite[Thm. 7.4a and 7.9a]{holm} with \v S\v toví\v cek \cite[Prop.~1.7]{stovicek13}
  we get the following generalization of (\ref{intro:add}).
If $\X$ is a generating set of objects in a Grothendieck abelian category, then 
\begin{align}
  \Phi(\sFilt\X )=\sFilt f_*(\X), \label{intro:sfilt}
\end{align}
where $\sFilt\X$ consists of all summands of $\X$-filtered objects.
In this paper we show that $\Phi(\X)$ can always be filtered by $\bigoplus f_*(\X)$ in the following sense.

\begin{bigthm}\label{thma}
  Let $\A$ be an $AB5$-abelian category, let $\X\subset\A$ and let $Q$ be a left-rooted quiver. Then
  \begin{enumerate}
  \item[i)]  Any $F\in\Phi(\mathscr{X})$ is $\bigoplus f_*(\mathscr{X})$-filtered. 
  \end{enumerate}
  If $\X$ is closed under filtrations, then
  \begin{enumerate}
  \item[ii)]   $\Phi(\X)=\Filt\bigoplus f_*(\X)$
  \end{enumerate}
  In particular we have the following.
  If $\X$ is a set, then
  \begin{enumerate}
  \item[iii)]   $\Phi(\Filt\X)=\Filt f_*(\X)=\Filt\Phi(\X)$
  \item[iv)]   $\Phi(\sFilt\X)=\sFilt f_*(\X)=\sFilt\Phi(\X)$
  \end{enumerate}
  If $X\subseteq FP_{2.5}(\A)$ and $\A$ is locally finitely presented, then
  \begin{enumerate}
  \item[v)]   $\Phi(\dlim\X)=\dlim\ext f_*(\X)=\dlim\Phi(\X)$
  \end{enumerate}
\end{bigthm}
Here $FP_{2.5}(\A)$ is a certain class of objects which sits between $FP_2(\A)$ and $FP_3(\A)$ with the property that it is always closed under extensions. In many situations (e.g $\A=\Rmod$) $FP_{2.5}(\A)=FP_2(\A)$ (Lemma~\ref{lemma4}).  

We note that $\varinjlim\ext\X=\varinjlim\add\X$ and $\Add\X=\sFilt\X$ when $\X$ consists of projective objects and that the finitely generated projective objects are $FP_n$ for any $n.$ Theorem~\ref{thma} is thus a generalization of (\ref{intro:add}) and (\ref{intro:dlim}). It also generalizes (\ref{intro:sfilt}) to not necessarily generating sets in arbitrary AB5-abelian category.
We show how to use this to reprove (\ref{intro:proj}) in abelian categories with enough projective objects. We also show (\ref{intro:flat}) (Lemma~\ref{fgprojchar}) when the category is generated by finitely generated projective objects and flat is understood as their direct limit closure (see Theorem~\ref{thmc} however for a more general version).

We then apply Theorem~\ref{thma} v) to the Gorenstein homological situation. We let $\GProj(\A)$  be the Gorenstein projective objects, let $\Gproj(\A)=\GProj(\A)\cap FP_{2.5}(\A)$ and immediately get
$\Phi(\varinjlim \Gproj(\A))=\varinjlim\ext f_*(\Gproj(\A)).$
Contrary to the case for ordinary projective objects, it is not clear, that this equals $\dlim\Gproj(\Rep(Q,\A))$ without some restrictions on $Q$.
In the following target-finite means that there are only finitely many arrows with a given target and locally path-finite means that there are only finitely many paths between two given vertices.
We have
\begin{bigthm}\label{thmb}
  Let $\A$ be a locally finitely presented category with enough projective objects, let $Q$ be a left-rooted quiver and assume that either
  \begin{itemize}
  \item   $Q$ is target-finite and locally path-finite, or
  \item  $\dlim \Gproj(\A)=\dlim \GProj(\A)$ (e.g if $\A=\Rmod$ and $R$ is Iwanaga-Gorenstein).
    \end{itemize}
  Then
    $$\Phi(\dlim \Gproj(\A))=\dlim \Gproj(\Rep(Q,\A))=\dlim\Phi(\Gproj(\A)).$$
In the latter case, this equals $\dlim \GProj(\Rep(Q,\A)).$
\end{bigthm}
Again contrary to the ordinary projective objects even for $\A=\Rmod$ it is not true in general that $\dlim \Gproj(\A)$ is all the Gorenstein Flat objects, $\GFlat(\A)$, nor those objects with Gorenstein injective Pontryagin dual, $\wGFlat(\A)$.
In the rest of the paper we study these classes in $\Rep(Q,\A).$
First we must explain what we mean by an abstract Pontryagin dual and we show how these arise natually and agree with the standard notion in well-known abelian categories.
We go on and characterize those objects with injective (or Gorenstein injective) Pontryagin dual as follows.

\begin{bigthm}\label{thmc}
  Let $\A$ be an abelian category with a Pontryagin dual to a category with enough injective objects and let $Q$ be a left-rooted quiver. Then
\begin{align*}
  \Flat(\Rep(Q,\A)) & = \Phi(\Flat(\A)) \\
  \wGFlat(\Rep(Q,\A)) & = \Phi(\wGFlat(\A))
\end{align*}
\end{bigthm}
Here $\Flat(\A)$ is those objects with injective Pontryagin dual so this result reproves (\ref{intro:flat}) using the simpler characterization of injective representations in Enochs, Estrada and Garc\'\i a Rozas \cite[Prop 2.1]{enochsinj} instead of going through the proof of (\ref{intro:dlim}) as in \cite{enochsflat}. \Tref{thmc} tells us that, under the conditions of \Tref{thmb}, if $\dlim \Gproj(\A)=\wGFlat(\A)$ then also $\dlim \Gproj(\Rep(Q,\A))=\wGFlat(\Rep(Q,\A)).$ (Corollary \ref{newprop20})

In \cite{enochsinj} it is proved that
$\wGFlat(\Rep(Q,\A))=\GFlat(\Rep(Q,\A))$ when $\A=\Rmod$ and $R$ is Gorenstein.
We end this paper by showing that this also hold if
$R$ is just assumed to be coherent if we impose proper finiteness conditions on $Q$.
\begin{bigthm}\label{thmd}
Let $R$ be a right coherent ring and let $Q$ be a left-rooted and target-finite quiver. Then
$$ \wGFlat(\Rep(Q,\Rmod))=\GFlat(\Rep(Q,\Rmod)).$$
\end{bigthm}
See also Proposition~\ref{propd} for a version for abelian categories.
If $Q$ is further locally path-finite (or $R$ is Gorenstein and $Q$ is just assumed to be left-rooted) the conditions for \Tref{thmb} and \Tref{thmc} are satisfied as well, so in this case (Corollary \ref{limgp=gf})
if $\dlim \Gproj(\A)=\GFlat(\A)$ then
\begin{align*}
  \dlim \Gproj(\Rep(Q,\A))
  =\GFlat(\Rep(Q,\A))=\Phi(\GFlat(\A)).
\end{align*}

The equality  $\varinjlim \Gproj(\Rmod) = \GFlat(\Rmod)$ is known to hold
when $R$ is an Iwanaga-Gorenstein ring (Enochs and Jenda \cite[Thm. 10.3.8]{relhomalg}) or if $R$ is an Artin algebra which is virtually Gorenstein (Beligiannis and Krause \cite[Thm. 5]{krause08}). In general $\varinjlim \Gproj(\Rmod)$ and $\GFlat(\Rmod)$ are different (Holm and Jørgensen \cite[Thm. A]{holm11}).

\section{Locally finitely presented categories}
In the following let $\mathscr{A}$ be an abelian category. First we recall some basic notions.

We say $\A$ is (AB4) if $\A$ is cocomplete and forming coproducts is exact,
(AB4${}^*$) if $\A$ is complete and forming products is exact, 
(AB5) if filtered colimits are exact,
Grothendieck if it is (AB5) and has a generator (i.e. a generating object or equivalently a generating set).
Here a class $\mathscr{S}\subseteq\mathscr{A}$ is said to generate $\mathscr{A}$ if
        it detects zero-morphisms i.e. a morphism $\begin{tikzcd} X \ar{r}{f} & Y \end{tikzcd}$
        is zero iff $\begin{tikzcd} S \ar{r}{g} & X \ar{r}{f} & Y \end{tikzcd}$ is zero for all
        $g$ with $S\in\mathscr{S}$.
\noindent

We write $X\in\varinjlim \mathscr{X}$ if $X=\varinjlim X_i$ for some filtered system $\{X_i\}\subseteq \mathscr{X}$.
We write $X\in\Filt\mathscr X$ if
there is a chain $X_0 \subseteq \ldots \subseteq X_\lambda = X$
for some ordinal $\lambda$ s.t. $X_{\alpha + 1}/X_\alpha \in \mathscr{X}$ for all $\alpha < \lambda$ and
$\varinjlim_{\alpha < \alpha_0}X_\alpha = X_{\alpha_0},$
for any limit ordinals $\alpha_0\leq \lambda.$
We say $X\in\Filt\mathscr{X}$ is \emph{$\X$-filtered}.
When $\lambda$ is finite, we say $X$ is a \emph{finite extension} of (objects of) $\mathscr X$,
and we let $\ext(\mathscr X)$ denote the class of finite extensions of $\mathscr X.$
This is also the extension closure of $\mathscr X$ i.e. the smallest subcategory of $\A$ containing $\X$ and closed under extensions.
For example the class $\bigoplus\X$ is the class of all (infinite) sums of elements of $\X.$
Such a sum, $\bigoplus_{i=1}^\lambda X_i$ is a colimit of a diagram with no arrows,
and as such is neither a direct limit nor a filtration.
It can however be realized as a filtration by  $\{\bigoplus_{i=1}^\alpha X_i\},$ for $\alpha<\lambda$ and as
a direct limit as $\{\bigoplus_{i\in I} X_i\},$ for $I$ finite, with arrows the inclusions.
In fact $\bigoplus\X=\Filt\X$ when $\X$ consists of projective objects. 
We say that $X\in\A$ is $FP_n$ if
the canonical map
\[ \varinjlim \Ext^k(X,Y_i) \to \Ext^k(X,\varinjlim Y_i)\] is an isomorphism for every
$0\leq k< n.$
The objects $FP_1(\A)$ are called finitely presented,
and the objects s.t the above map is injective for $k=0$ is called finitely generated and denoted $FP_0(\A).$
The category $\A$ is called locally finitely presented
  if it satisfies one (and therefore all) of the following equivalent conditions:
  \begin{enumerate}[(i)]
    \item $FP_1(\mathscr{A})$ is skeletally small (i.e. the isomorphism classes form a set)
    and $\varinjlim FP_1(\mathscr{A})=\mathscr{A}$ (Crawley-Boevey \cite{boevey94})
    \item $\mathscr{A}$ is Grothendieck and $FP_1(\mathscr{A})$ generate $\mathscr{A}$. (Breitsprecher \cite{breit70})
    \item $\mathscr{A}$ is Grothendieck and $\varinjlim FP_1(\mathscr{A})=\mathscr{A}$ (\cite{breit70}).
    \end{enumerate}

    The direct limit is very well-behaved in locally finitely presented categories. In particular we have that if $\X\subseteq FP_1(\A)$ is closed under direct sums, then $\varinjlim\X$ is closed under direct limits, and is thus the direct limit closure of $\X$ \cite[Lemma p. 1664]{boevey94}.
    We also have the following. The proof was communicated to me by Jan \v S\v toví\v cek (any mistakes are mine).

\begin{prop}\label{prop1}
Let $\mathscr{A}$ be a locally finitely presented abelian category.
If $\mathscr{X}\subseteq FP_2(\mathscr{A})$ is closed under extensions then so is $\varinjlim \mathscr{X}$.
It is thus closed under filtrations.
\end{prop}
\begin{proof}
Let $\{S_i\}, \{T_j\}\subseteq \mathscr{X}$ be directed systems and let
\[0\to \varinjlim S_i \to E \to \varinjlim T_j \to 0 \]
be an exact sequence. We want to show that $E\in \varinjlim \mathscr{X}$.
First by forming the pullback
\[
\begin{tikzcd}
0
\ar[r]  &
\varinjlim S_i
\ar[r]
\ar[d]  &
E_j
\ar[phantom, very near start]{dr}{\ulcorner}
\ar[r]
\ar[d] &
T_j
\ar[r]
\ar[d] &
0 \\
0
\ar[r] &
\varinjlim S_i
\ar[r] &
E
\ar[r] &
\varinjlim T_j
\ar[r] &
0
\end{tikzcd}
\]
we see that $E=\varinjlim E_j$ since $\mathscr{A}$ is AB5
as it is locally finitely presented abelian, hence Grothendieck. Now since $T_j$ is in $FP_2(\mathscr{A})$ for every $j$
we have that
\[
  [0 \to \varinjlim S_i \to E_j \to T_j \to 0] \in\Ext^1(T_j,\varinjlim S_i)
\]
is in the image of the canonical map from $\varinjlim\Ext^1(T_j,E_i)$, that is, it is a pushout
\[
\begin{tikzcd}
0
\ar{r}
& S_i
\ar{r}
\ar{d}
& E_{ij}
\ar{r}
\ar{d}
& T_j
\ar{d}
\ar{r}
& 0 \\
0
\ar{r}
& \varinjlim S_i
\ar{r}
& E_j
\ar[phantom, very near start]{ul}{\lrcorner}
\ar{r}
& T_j
\ar{r}
& 0
\end{tikzcd}
\]
for some $i$ and some extension $E_{ij}\in\A.$

Now construct for every $k\geq i$ the pushout
\[
\begin{tikzcd}
0
\ar{r}
& S_i
\ar{r}
\ar{d}
& E_{ij}
\ar{r}
\ar{d}
& T_j
\ar{d}
\ar{r}
& 0 \\
0
\ar{r}
& S_k
\ar{r}
& E_{kj}
\ar[phantom, very near start]{ul}{\lrcorner}
\ar{r}
& T_j
\ar{r}
& 0
\end{tikzcd}
\]
Then $\varinjlim_k E_{kj} = E_j$ so
$E_j \in \varinjlim \mathscr{X}$ as $E_{kj}\in\X$ when $\mathscr{X}$ is closed under extensions.

Finally $E=\varinjlim E_j \in \varinjlim \mathscr{X}$ as
$\dlim\X$ is closed under direct limits when $\mathscr{X}\subset FP_1(\A).$
\end{proof}

The classes $FP_n(\mathscr{A})$ are all closed under finite sums (as in \cite[Lem. 1.3]{breit70}). They are not necessarily closed under extensions, but the following subclasses are:

\begin{defn}
Let $\mathscr{A}$ be an abelian category. We say $X\in\mathscr{A}$ is $FP_{n.5}$ if $X$ is $FP_n$
and furthermore, that the natural map $\varinjlim \Ext^n(X,Y_i)\rightarrow \Ext^n(X,\varinjlim Y_i)$ is monic
for every filtered system $\{Y_i\}\subseteq\A$.
We let $FP_*$ stand for an unspecified (but fixed) $FP_n$ or $FP_{n.5}$
\end{defn}
Note that by definition $FP_0(\A)=FP_{0.5}(\A)$ and also
$FP_1(\A)=FP_{1.5}(\A)$ by Stenström \cite[Prop. 2.1]{stenstrom70} when $\A$ is AB5. We have the following generalization of \cite[Lem. 1.9]{breit70} for $n,*=1$ and $\A$ Grothendieck.

\begin{lemma}\label{lemmadelta}
  Let $\A$ be an AB5-abelian category and let $$0\to A\to B\to C\to 0$$ be an exact sequence. Then
\begin{enumerate}[(i)]
  \item If $A$ and $C$ are $FP_{n.5}$, then so is $B$.
  \item If $B$ is $FP_*$ then $A$ is $FP_{*-1}$ iff $C$ is $FP_*$.
\end{enumerate}

\end{lemma}

\begin{proof} $(i)$ Let $\{X_i\}\subset\A$ be a filtered system. From the long exact sequence in homology we get for all $k<n:$
\[
\begin{tikzcd}[column sep = tiny]
\varinjlim \Ext^{k-1} (A,X_i)
\ar{r}
\ar{d}{\cong}
& \varinjlim \Ext^k(C, X_i)
\ar{r}
\ar{d}{\cong}
& \varinjlim \Ext^k(B, X_i)
\ar{r}
\ar{d}
& \varinjlim \Ext^k(A, X_i)
\ar{r}
\ar{d}{\cong}
& \varinjlim \Ext^{k+1}(C,X_i)
\ar[hookrightarrow]{d}
\\
\Ext^{k-1}(A, \varinjlim X_i)
\ar{r}
& \Ext^k(C, \varinjlim X_i)
\ar{r}
& \Ext^k(B, \varinjlim X_i)
\ar{r}
& \Ext^k(A,\varinjlim X_i)
\ar{r}
& \Ext^{k+1}(C, \varinjlim X_i)
\end{tikzcd}
\]
and
\[
\begin{tikzcd}
\varinjlim \Ext^{n-1}(A, X_i)
\ar{r}
\ar{d}{\cong}
& \varinjlim \Ext^n(C,X_i)
\ar{r}
\ar[hookrightarrow]{d}
& \varinjlim \Ext^n(B,X_i)
\ar{r}
\ar{d}
& \varinjlim \Ext^n(A, X_i)
\ar[hookrightarrow]{d}
\\
\Ext^{n-1}(A,\varinjlim X_i)
\ar{r}
& \Ext^n(C, \varinjlim X_i)
\ar{r}
& \Ext^n(B, \varinjlim X_i)
\ar{r}
& \Ext^n(A, \varinjlim X_i)
\end{tikzcd}
\]

And the result follows by the 5-lemma.
$(ii)$ is proved similarly. Note that when $*=1$ we must use that $FP_1=FP_{1.5}$ because $FP_0=FP_{0.5}.$
\end{proof}

  \begin{lemma}\label{lemma4}
    Let $\A$ be an AB5-abelian category generated by a set
    of $FP_{n.5}$-objects. Then
    
    \begin{enumerate}[(i)]
    \item If $X\in FP_0(\A)$ there exists an epi
      $X_0\to X$ with $X_0\in FP_{n.5}(\A).$ 
    \item $FP_k(\A) = FP_{k.5}(\A)$ for all $k\leq n$
    \end{enumerate}
  \end{lemma}
  \begin{proof}
    For $(i)$ notice that by \cite[satz 1.6]{breit70}  if $\A$ is generated by $\mathscr{X}\subseteq FP_1(\A)$
  and $C\in FP_0(\A)$ then we have an epi from a finite sum of elements of $\mathscr{X}$ to $C$.  But $FP_n$ (and $FP_{n.5}$) are all closed under finite sums.
  The proof of $(ii)$ goes by induction.
  The case $n=0$ is true by definition, so assume $\A$ is generated by a set of $FP_{n.5}$-objects and that $X\in\ FP_n(\mathscr{A})$.
  By $(i)$ we get an exact sequence
\[
  \begin{tikzcd} 0\ar{r} & X_1 \ar{r} & X_0 \ar{r} & X \ar{r} & 0\end{tikzcd}
  \]  with $X_0\in FP_{n.5}(\A).$
  By \Lref{lemmadelta} $(ii)$ $X_1\in FP_{n-1}(\A)$ which by induction hypothesis equals $FP_{(n-1).5}(\A)$ so $X\in FP_{n.5}(\A)$ again by \Lref{lemmadelta} $(ii)$.
  \end{proof}

  In particular $FP_{n.5}(\Rmod)=FP_n(\Rmod)$ is closed under extensions for any $n$ and any ring $R.$
  We think of the objects of $FP_*(\A)$ as beeing small.

\section{Quiver representations}
Let $Q$ be a quiver, i.e. a directed graph.
We denote the vertices by $Q_0$ and we denote an arrow (resp. a path) from $w$ to $v$ by
$w \to v$ (resp. $w \leadsto v$).
A quiver may have infinitely many vertices and arrows, but we will need the following finiteness conditions.
\begin{defn}
  Let $Q$ be a quiver.
  We say $Q$ is \emph{target-finite} (resp. \emph{source-finite}) if there are only finitely many arrows with a given target (resp. source).
  We say $Q$ is \emph{left-rooted} (resp. \emph{right-rooted}) if there is no infinite sequence of composable arrows
  $\cdots\to\bullet\to\bullet$ (resp.   $\bullet\to\bullet\to\cdots$ ).
  Finally we say $Q$ is \emph{locally path-finite} if there is only finitely many paths between any two given vertices.
\end{defn}
\begin{remark}\label{rem:finitequiver}
  Notice that $Q$ is target-finite (resp. left-rooted) iff $Q^{\text{op}}$ is source-finite (resp. right-rooted) and that left/right-rooted quivers are necessarily acyclic (i.e have no cycles or loops). Locally path-finite is self-dual.
Even if a quiver satisfies all of the above finiteness conditions, it can still have infinitely many vertices and arrows, e.g the quiver
 $\cdots \leftarrow\bullet\to\bullet\leftarrow\bullet\to\bullet\leftarrow\bullet\to\cdots$
\end{remark}
When the quiver is left-rooted we can
use the following sets for inductive arguments.
Let $V_0=\emptyset$ and define for any ordinal $\lambda$,
$V_{\lambda + 1} = \{v \in Q_0 | w \to v \Rightarrow w \in V_\lambda \}$ 
and for limit ordinals $V_\lambda = \bigcup_{\alpha < \lambda} V_\alpha$.
Notice that $V_1$ is precisely the sources of $Q.$

As noted in \cite[Prop. 3.6]{enochsflat} a quiver is left-rooted precisely when $Q_0=V_\lambda$ for some $\lambda.$

\begin{ex}
Let $Q$ be the (left-rooted) quiver:
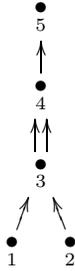
\begin{displaymath}
\xymatrix@R=1pc@C=0pc{
{} & \underset{5}{\bullet} & {}
\\
{} & \underset{4}{\bullet} \ar[u] & {}
\\
{} & \underset{3}{\bullet} \ar@<0.5ex>[u] \ar@<-0.5ex>[u] & {}
\\
\underset{1}{\bullet} \ar[ur] & {} & \underset{2}{\bullet} \ar[ul]
}
\end{displaymath}
For this quiver, the transfinite sequence $\{V_\alpha\}$ looks like this:
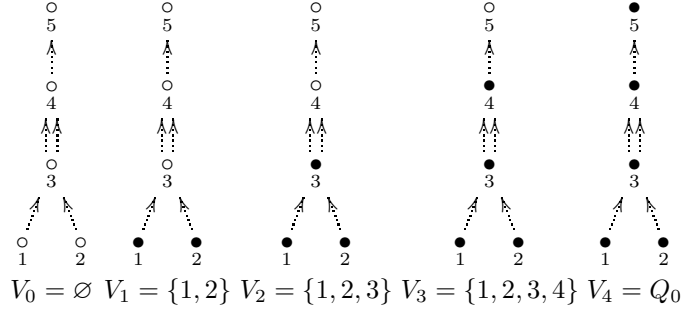
\begin{displaymath}
\begin{array}{c@{~~~~~}c@{~~~~~}c@{~~~~~}c@{~~~~~}c}
    {
\xymatrix@R=1pc@C=0pc{
{} & \underset{5}{\circ} & {}
\\
{} & \underset{4}{\circ} \ar@{..>}[u] & {}
\\
{} & \underset{3}{\circ} \ar@{..>}@<0.5ex>[u] \ar@{..>}@<-0.5ex>[u] & {}
\\
\underset{1}{\circ} \ar@{..>}[ur] & {} & \underset{2}{\circ} \ar@{..>}[ul]
}
    }
        &
    {
\xymatrix@R=1pc@C=0pc{
{} & \underset{5}{\circ} & {}
\\
{} & \underset{4}{\circ} \ar@{..>}[u] & {}
\\
{} & \underset{3}{\circ} \ar@{..>}@<0.5ex>[u] \ar@{..>}@<-0.5ex>[u] & {}
\\
\underset{1}{\bullet} \ar@{..>}[ur] & {} & \underset{2}{\bullet} \ar@{..>}[ul]
}
    }
        &
    {
\xymatrix@R=1pc@C=0pc{
{} & \underset{5}{\circ} & {}
\\
{} & \underset{4}{\circ} \ar@{..>}[u] & {}
\\
{} & \underset{3}{\bullet} \ar@{..>}@<0.5ex>[u] \ar@{..>}@<-0.5ex>[u] & {}
\\
\underset{1}{\bullet} \ar@{..>}[ur] & {} & \underset{2}{\bullet} \ar@{..>}[ul]
}
    }
        &
    {
\xymatrix@R=1pc@C=0pc{
{} & \underset{5}{\circ} & {}
\\
{} & \underset{4}{\bullet} \ar@{..>}[u] & {}
\\
{} & \underset{3}{\bullet} \ar@{..>}@<0.5ex>[u] \ar@{..>}@<-0.5ex>[u] & {}
\\
\underset{1}{\bullet} \ar@{..>}[ur] & {} & \underset{2}{\bullet} \ar@{..>}[ul]
}
    }
        &
    {
\xymatrix@R=1pc@C=0pc{
{} & \underset{5}{\bullet} & {}
\\
{} & \underset{4}{\bullet} \ar@{..>}[u] & {}
\\
{} & \underset{3}{\bullet} \ar@{..>}@<0.5ex>[u] \ar@{..>}@<-0.5ex>[u] & {}
\\
\underset{1}{\bullet} \ar@{..>}[ur] & {} & \underset{2}{\bullet} \ar@{..>}[ul]
}
    }
        \\
V_0 = \varnothing
& V_1 = \{1,2 \}
& V_2 = \{1,2,3 \}
& V_3 = \{1,2,3,4 \}
& V_4 = Q_0
\end{array}
\end{displaymath}
\end{ex}
Let now $\mathscr{A}$ be an abelian category.
A quiver $Q$ generates a category $\overline{Q}$, called the path category,
with objects $Q_0$ and morphisms the paths in $Q$. We define
$\textnormal{Rep}(Q,\mathscr{A}) = \textnormal{Fun}(\overline{Q}, \mathscr{A})$.
Note that $F\in \textnormal{Rep}(Q,A)$ is given by its values on vertices and arrows
and we picture $F$ as a $Q$-shaped diagram in $\mathscr{A}$.

For  $v\in Q_0$ the evaluation functor $e_v : \Rep(Q,\mathscr{A}) \to \mathscr{A}$ is given by $e_v(F) = F(v)$ for $v\in Q_0$ and $e_v(\eta)=\eta_v$ for $\eta\colon F\to G$.
If $\mathscr{A}$ has coproducts (or $Q$ is locally path-finite) this has a left-adjoint $f_v : \mathscr{A} \to \Rep(Q,\mathscr{A})$ given by
\[f_v (X)(w) = \bigoplus_{v \leadsto w} X\]
where the sum is over all paths from $v$ to $w$
and $f_v(X)(w \to w') $ is the natural inclusion.
For $\mathscr{X}\subseteq \mathscr{A}$
we define
\[f_*(\mathscr{X})=\{f_v(X) \mid v\in Q_0, X\in \mathscr{X}\}\subseteq\Rep(Q,\A).\]
See \cite{enochsflat} or \cite{holm} for details.

\begin{remark}\label{delta}
Limits and colimits are point-wise in $\Rep(Q,\mathscr{A})$, so $e_v$ preserves them
and is in particular exact. Thus its left-adjoint $f_v$ preserves projective objects.
\end{remark}

\begin{defn}\label{defphi}
  For any quiver $Q$, any abelian category $\mathscr{A}$, any $F\in\Rep(Q,\A)$ and any $v\in Q_0$ we have a canonical map
  $\varphi_v^F = \bigoplus_{w\to v} F(w)\to F(v)$
    and we set
  \begin{align*}
  \Phi(\mathscr{X})=\left\{ F \in \Rep(Q,\mathscr{A}) \middle|
   \forall v\in Q_0: \varphi_v^F \textnormal{ is monic and } \textnormal{coker}\,\varphi_v^F\in\mathscr{X}\right\}.
   \end{align*}
\end{defn}

\begin{remark}\label{fviso}
Observe that
$f_v(\mathscr{X})\subseteq \Phi(\mathscr{X}).$ 
In fact for any $v\in Q_0$,
$\varphi_w^{f_v(X)}$
is an isomorphism, unless $w = v$ in which case it is monic (in fact zero if $Q$ is acyclic) with cokernel $X$.
As in \cite[Prop. 7.3]{holm} if $Q$ is left-rooted then $\Phi(\X)\subseteq\Rep(Q,\X)$ if $\X$ is closed under arbitrary sums or $Q$ is locally path-finite and $\X$ is closed under finite sums.
\end{remark}
The aim of this section is to show that sums of objects of $f_*(\X)$ filter $\Phi(\X).$
Let us first see how $f$ and $\Phi$ play together with various categorical constructions.
\begin{lemma}\label{nylem1}
  Let $Q$ be a quiver, $\mathscr{A}$ an abelian category satisfying AB4, and
  $\mathscr{X}\subseteq \mathscr{A}$  arbitrary. Then
  \begin{enumerate}[(i)]
    \item $f_*(\textnormal{extensions of } \mathscr{X}) \subseteq \textnormal{extensions of } f_*(\mathscr{X})$,
    \item $f_*(\textnormal{summands of }\mathscr{X}) \subseteq \textnormal{summands of } f_*(\mathscr{X})$,
    \item $f_*(\varinjlim \mathscr{X}) \subseteq \varinjlim f_*(\mathscr{X})$,
    \item $f_*(\Filt \mathscr{X}) \subseteq \Filt f_*(\mathscr{X})$.
  \end{enumerate}
\end{lemma}

\begin{proof}
$(i)$ follows since $f_v$ is exact when $\mathscr{A}$ is AB4 and
$(iii)$ since $f_v$ is a left adjoint.
$(ii)$ is clear and $(iv)$ follows from $(i)$ and $(iii)$.
\end{proof}

\begin{lemma}\label{nylem2}
  Let again $Q$ be a quiver, $\mathscr{A}$ an abelian category satisfying AB4, and
  $\mathscr{X}\subseteq \mathscr{A}$ arbitrary. Then
    \begin{enumerate}[(i)]
    \item $ \Phi (\textnormal{extensions of }\mathscr{X})\subseteq \textnormal{extensions of } \Phi(\mathscr{X})$,
    \item $\textnormal{summands of }\Phi (\mathscr{X}) \subseteq \Phi (\textnormal{summands of }\mathscr{X}) $.
    \end{enumerate}
    When $\mathscr{A}$ is AB5 we further have
    \begin{enumerate}[(i)]
    \setcounter{enumi}{2}
    \item $\varinjlim \Phi (\mathscr{X}) \subseteq \Phi (\varinjlim \mathscr{X})$,
    \item $\Filt \Phi (\mathscr{X}) \subseteq \Phi(\Filt \mathscr{X}).$
    \end{enumerate}
        When $\mathscr{A}$ is AB4${}^*$ and $Q$ is target-finite we have
    \begin{enumerate}[(i)]
    \setcounter{enumi}{4}
    \item $\prod \Phi (\mathscr{X}) \subseteq \Phi (\prod \mathscr{X})$.
    \end{enumerate}
\end{lemma}

\begin{proof}
$(ii)$ follows as retracts respects kernels and cokernels, $(iii)$ is clear when $\mathscr{A}$ satisfies AB5. For $(i)$
let $0\to F\to F\pprime\to F' \to 0$, be an exact sequence with $F, F'\in \Phi(\mathscr{X})$. For every $v \in Q_0$ we have that

\begin{equation*}
\begin{tikzcd}
    {}
    & 0
        \arrow[d]
    & {}
    & 0
        \arrow[d] \\
    0
        \arrow[r]
    & \oplus_{w\to v} F(w)
        \arrow[r]
        \arrow[d]
    & \oplus_{w \to v} F\pprime(w)
        \arrow[r]
        \arrow[d]
    & \oplus_{w\to v} F'(w)
        \arrow[r]
        \arrow[d]
    & 0 \\
    0
        \arrow[r]
    & F(v)
        \arrow[r]
        \arrow[d]
    & F\pprime(v)
        \arrow[r]
    & F'(v)
        \arrow[r]
        \arrow[d]
    & 0 \\
    {}
    & C
        \arrow[d]
    &
    & C'
        \arrow[d] \\
    {}
    & 0
    &
    & 0
\end{tikzcd}
\end{equation*}
 has exact rows since $\mathscr{A}$ is AB4 and $e_v$ is exact.
  The condition follows from the snake lemma, since $C,C'\in \mathscr{X}$.

Again $(iv)$ follows from $(i)$ and $(iii)$.
For $(v)$ we notice that for any $\{F_i\}\subset\A$ and vertex $v$ we have
$\prod_i\phi^{F_i}_v=\phi^{\prod F}_v$ since the sum in the definition of $\phi$ is finite, hence a product, when $Q$ is target-finite.
\end{proof}

As for smallness we have the following

\begin{lemma}\label{star1}
Let $\mathscr{A}$ be an abelian category.

\begin{enumerate}[(i)]
    \item If $\mathscr{A}$ satisfies AB5 then  $f_v$ preserves $FP_*$
    \item If $Q$ is locally path-finite, then $e_v(-)$ preserves $FP_*$.
     \item If $Q$ is target-finite and locally path-finite then
    \[ \Phi(\mathscr{X})\cap FP_*(\textnormal{Rep}(Q,\mathscr{A})) \subseteq \Phi(\mathscr{X}\cap FP_*(\mathscr{A})). \]
\end{enumerate}
\end{lemma}

\begin{proof}
$(i)$
This follows from the natural isomorphism (\cite[prop 5.2]{holm})
\[\Ext^i(f_v(X), -) \cong \Ext^i(X, e_v(-))\]
and the fact that $e_v$ preserves filtered colimits (\Rref{delta}).

$(ii)$
In this case $e_v$ has a right adjoint $g_v(X)(w) = \prod_{w\leadsto v} X$ (see \cite[3.6]{holm})
which is a finite product, hence a sum, as $Q$ is locally path-finite. So $g_v(-)$ preserves filtered colimits.
Thus $e_v$ preserves $FP_*$, by the natural isomorphism (\cite[prop 5.2]{holm})$$\Ext^1(e_v(X),-)\cong \Ext^1(X, g_v(-))$$ $(iii)$ Let $F\in\Phi(\X)$ be $FP_*.$ Given $v\in Q_0$ we only need to show that  coker$\phi^F_v$ is $FP_*.$ Since $Q$ is target-finite, $\oplus_{w\to v}F(w)$ is a finite sum of $FP_*$-objects by $(ii)$ and since $FP_*$ is closed under finite sums the result follows from $(ii)$ and \Lref{lemmadelta} $(ii)$.
\end{proof}
The following two lemmas will be used to construct a $\oplus f_*(\mathscr{X})$-filtration for any $F\in \Phi (\mathscr{X})$ for suitable $\X\subset\A$ when $Q$ is left-rooted. This is the key in proving \Tref{thma}.

\begin{lemma}\label{lem3}
  Let $Q$ be an acyclic (e.g. left-rooted) quiver and $\mathscr{A}$ an abelian category satisfying AB4.
  If $F\in\Phi(\mathscr{X})$ there exists a subrepresentation $F'\subseteq F$ such that
  \begin{enumerate}
    \item[(a)] $F' \in  \bigoplus f_*(\mathscr{X})$,
    \item[(b)] $F'(v) = F(v) \quad \forall u\in V^F = \{v\in Q_0 | w \to v \Rightarrow F(w) = 0\}$,
    \item[(c)] $F / F' \in \Phi(\X)$, with $\coker\phi^{F/F'}_v=\coker\phi^F_v$ when $v\not\in V^F.$
  \end{enumerate}
\end{lemma}

\begin{proof}
  Define $F' = \bigoplus_{v \in V^F} f_v(F(v))$. We wish to prove that $F'$ is a subrepresentation and that it
  satisfies \mbox{$(a)$-$(c)$}.

  Clearly $F'$ satisfy $(a).$ To see (b) it suffices to prove, that for any non-trivial path
  $w\leadsto v$ with $v\in V^F$ we have $F(w)=0$ - 
  because then for any $v\in V^F$ we have $f_v(F(v))(v)=F(v)$ and
  $f_w(F(w))(v)=0,$ $w\neq v.$
  So let $v\in V^F$ and assume there is a path $w\overset{p}{\leadsto} w'\to v.$ Then $F(w')=0$ as $v\in V^F.$ But then also $F(w)=0$ as $F(p)$ is monic since $F\in\Phi(\X).$

  To see that $F'$ is a subrepresentation satisfying (c)
  we use the map $F'\to F$ induced by the counits $f_ve_v(F)\to F.$ If $v$ is not reachable from $V^F$ (i.e. there is no path $w\leadsto v$ with $w\in V^F$) this is trivial since then $F'(v) = 0$.
  
  So let $Q'$ be the subquiver consisting of all vertices $Q_0'$ reacheable from $V^F$ (i.e. $Q_0'=\{v\in Q_0 \mid \exists w\leadsto v, w\in V^F\}$ with arrows $\{w\to v\mid w\in Q_0'\}$). 
  We want for all $V\in Q'_0$ that there are exact sequences
\begin{enumerate}
    \item[(1)] $0 \to F'(v) \to F(v)$
    \item[(2)] $0 \to \bigoplus_{w \to v} F/F' (w) \to F/F'(v) \to \coker\phi^F_v\to 0$ when $v\not\in V^F.$
\end{enumerate}
  Since $Q$ is acyclic, $Q'$ is left-rooted with sources $V^F.$ 
  We can thus proceed by induction on the sets $V_\lambda'.$
  The case $v\in V'_1=V_F$ is taken care of by (b), so assume $(1)$ for all $w\in V'_\alpha$ and all $\alpha < \lambda\neq1$, and let $v\in V'_\lambda$. Then we have
the following commutative diagram
with exact rows and columns
\begin{equation*}
\begin{tikzcd}
    {}
    & 0
        \arrow[d] \\
    0
        \arrow[r]
    & \oplus_{w\to v} F'(w)
        \arrow[r]
        \ar[d]
    & F'(v)
        \arrow[r]
        \ar{d}
    & 0 \\
    0
        \arrow[r]
    & \oplus_{w\to v} F(w)
        \arrow[r]
        \ar[d]
    & F(v)
        \arrow[r]
    & \coker\phi^F_v
        \arrow[r]
    & 0 \\
    & \oplus_{w\to v} F/F'(w)
    \ar[d]
    &
    &
    & \\
    & 0
    &
    &
    &
\end{tikzcd}
\end{equation*}
The first row is exact as $v\notin V^F$ (see Remark~\ref{fviso}), the second as $F\in\Phi(\X)$ and the first column by induction hypothesis and the assumption that $\A$ is $AB4.$
Now $(1)$ and $(2)$ follows for $v\in V_\lambda$ by the snake lemma.
\end{proof}

\begin{lemma}\label{lem4}
Let $Q$ be an acyclic quiver, $\mathscr{A}$ an AB5-abelian category,
and let $\mathscr{X}\subseteq\A$.
Then for any $F\in \Phi(\mathscr{X})$ there exists a chain
  $0= F_0\subseteq F_1\subseteq \ldots \subseteq F_\lambda \subseteq \ldots\subseteq F$ of subrepresentations of $F$,
  such that for all ordinals $\lambda$
  \begin{enumerate}
    \item[(a)] $F_{\lambda}/F_\alpha \in \bigoplus f_*(\mathscr{X})$, if                                                $\lambda = \alpha + 1$
    \item[(b)] $F_\lambda(v) = F(v)\, \textnormal{for all}\, v \in \bigcup_{\alpha<\lambda} V^{F/F_\alpha}$
    \item[(c)] $F / F_\lambda \in \Phi(\mathscr{X})$ with $\coker\phi^{F/F_\lambda}=\coker\phi^F_v$ for $v\not\in \bigcup_{\alpha<\lambda} V^{F/F_\alpha}.$
  \end{enumerate}
Notice that $ \bigcup_{\alpha<\beta+1} V^{F/F_\alpha}=V^{F/F_\beta}$
\end{lemma}

\begin{proof}
We will construct such a filtration by transfinite induction.
$0=F_{0}$ is evident so assume $F_\alpha$ satisfying $(a)$-$(c)$ has been
constructed for all $\alpha < \lambda$

If $\lambda = \alpha + 1$ then by Lemma~\ref{lem3}
we have an $F'\subseteq F/F_\alpha$ s.t. $F' \in \bigoplus f_*(\mathscr{X})$ and s.t.
$F\pprime= (F/F_\alpha)/F'\in \Phi(\mathscr{X})$
satisfies
\[F\pprime(v) = 0 \textnormal{ for all } v\in V^{F/F_\alpha} = \{v\in Q_0 |w\to v \implies F(w) = F_\alpha (w)\}.\]
and $$\coker\phi^{F''}_v=\coker\phi^{F/F_\alpha}_v
=\coker\phi^{F}_v
\text{ for all } v\not\in V^{F/F_\alpha}$$

Now let $F_\lambda$ be the pullback
\[ \begin{tikzcd}
F_\lambda
\ar[phantom, very near start]{dr}{\ulcorner}
\ar[r]
\ar[d] &
F'
\ar[d] \\
F
\ar[r] &
F/F_\alpha
\end{tikzcd} \]
Then $(a)$ follows as
$F_\lambda / F_\alpha \cong F'$ and $(b)$ and $(c)$ follows since
$F/F_\lambda \cong F\pprime$.

If $\lambda$ is a limit ordinal, we set
\[F_\lambda = \bigcup_{\alpha < \lambda} F_\alpha\]
so that
$$F(v)=F_\lambda(v)\text{ when }
v\in \bigcup_{\alpha<\lambda}V^{F/F_\alpha}$$
Then $(a)$ is void and we get $(b)$ by noting that
when $v\in V^{F/F_\alpha}$ for some $\alpha<\lambda,$ then $F_\lambda(v)$ is the limit of a filtration eventually equal to $F(v)$
$$F_\lambda(v)=e_v\left(\bigcup_{\alpha < \lambda} F_\alpha\right)=\bigcup_{\alpha<\lambda} F_\alpha(v)=F(v).$$
To prove $(c)$ we similarly notice that
$\phi_v^{F/\dlim F_\alpha}=\dlim\phi_v^{F/F_\alpha}$ is monic for any vertex $v$ as $\A$ is AB5 and when $v\not\in \bigcup_{\alpha<\lambda}V^{F/F_\alpha}$ then
$$\coker\phi_v^{F/\dlim F_\alpha}
=\dlim\coker\phi^{F/F_\alpha}_v=  \dlim\coker\phi^F_v=\coker\phi^F_v$$
\end{proof}
The following figure shows an example of this construction.
\begin{figure}[htb]
\begin{displaymath}
\begin{array}{c@{~~~~~}c@{~~~~~}c@{~~~~~}c@{~~~~~}c}
    {
\xymatrix@R=1.1pc@C=-1.6pc{
{} & (x\oplus y\oplus z_0)^2\oplus z_1 & {}
\\
{} & (x\oplus y\oplus z_0)^2\oplus z_1 \ar[u] & {}
\\
{} & x\oplus y\oplus z_0 \ar@<0.5ex>[u] \ar@<-0.5ex>[u] & {}
\\
x \ar[ur] & {} & y \ar[ul]
}
    }
        &
    {
\xymatrix@R=1.1pc@C=-1.6pc{
{} & (x\oplus y)^2 & {}
\\
{} & (x\oplus y)^2 \ar[u] & {}
\\
{} & x\oplus y \ar@<0.5ex>[u] \ar@<-0.5ex>[u] & {}
\\
x \ar[ur] & {} & y \ar[ul]
}
    }
        &
    {
\xymatrix@R=1.1pc@C=-1.6pc{
{} & z_0^2\oplus z_1 & {}
\\
{} & z_0^2 \oplus z_1 \ar[u] & {}
\\
{} & z_0 \ar@<0.5ex>[u] \ar@<-0.5ex>[u] & {}
\\
0 \ar[ur] & {} & 0 \ar[ul]
}
    }
        &
    {
\xymatrix@R=1.1pc@C=-1.6pc{
{} & (x\oplus y\oplus z_0)^2 & {}
\\
{} & (x\oplus y\oplus z_0)^2 \ar[u] & {}
\\
{} & x\oplus y\oplus z_0 \ar@<0.5ex>[u] \ar@<-0.5ex>[u] & {}
\\
x \ar[ur] & {} & y \ar[ul]
}
    }
        &
    {
\xymatrix@R=1.1pc@C=-1.6pc{
{} & z_1 & {}
\\
{} & z_1 \ar[u] & {}
\\
{} & 0 \ar@<0.5ex>[u] \ar@<-0.5ex>[u] & {}
\\
0 \ar[ur] & {} & 0 \ar[ul]
}
    }
        \\
F
& F_1
& F/F_1
& F_2
& F/F_2
\end{array}
\end{displaymath}
\ \\
\begin{displaymath}
   F_3=F 
\end{displaymath}
\caption{Example of the construction of the subrepresentations $F_\alpha$}
\end{figure}
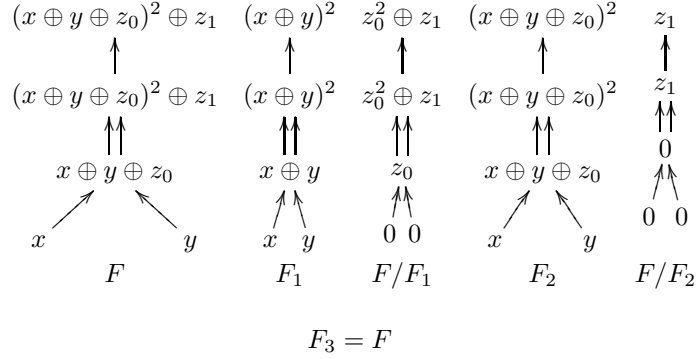

We can now proof Theorem~\ref{thma} from the introduction.
\begin{proof}[Proof of Theorem~\ref{thma}]\ 
\begin{enumerate}[$i)$]
 \item Let $F\in\Phi(\X)$ and let $\{F_\lambda\}$ be the filtration of Lemma~\ref{lem4}.
  First we show that $F_\lambda(v)=F(v)$ for all $v\in V_\lambda.$ The case $\lambda=0$ is trivial, so
  let $\lambda=\alpha+1,$ assume $F_\alpha(v)=F(v).$ and let $v\in V_\lambda.$ Then for paths $w\to v$
  we have $w\in V_\alpha$ so $F_{\alpha}(w)=F(w).$
  This precisely says that $v\in V^{F/F_\alpha}$ i.e  $F_\lambda(v)=F(v).$
  If $\lambda$ is a limit ordinal then
  $F_\lambda=\bigcup_{\alpha<\lambda} F_\lambda$ so
  $F_\lambda(v)=F(v)$ when $v\in\bigcup_{\alpha<\lambda}V_\alpha=V_\lambda.$
  Now since $Q$ is left-rooted, $F=F_\lambda$ for some $\lambda.$

  \item When $\X$ is closed under filtrations we have
  $$\Filt\bigoplus f_*(\X)\subsetref{Rem.~\ref{fviso}}
  \Filt\bigoplus \Phi(\X)\subsetref{Lem.~\ref{nylem2}}
  \Phi(\Filt\bigoplus \X)\subseteq \Phi(\X).$$
  \item When $\X$ is a set, $\Filt f_*(\X)$ is closed under filtrations \cite[Lem. 1.6]{stovicek13} hence
\begin{align*}
  \Filt\Phi(\X)
  &\subsetref{Lem.~\ref{nylem2}}\Phi(\Filt\X)
  \subseteq \Filt\bigoplus f_*(\Filt\X) \\
  &\subsetref{Lem.~\ref{nylem1}}\Filt\bigoplus\Filt f_*(\X)
   \subseteq\Filt f_*(\X)\\
  &\subsetref{Rem.~\ref{fviso}}\Filt\Phi(\X)
\end{align*}
\item This is proven similar to $iii)$. Just observe that a filtration of summands is a summand of a filtration.
  \item When $\X$ is $FP_{2.5}$ then
  $f_*(\X)$ is $FP_{2.5}$ by Lemma~\ref{star1}
  and so is $\ext f_*(\X)$ by Lemma~\ref{lemmadelta}. Hence
  $\dlim \ext f_*(\X)$ is closed under extensions by Proposition~\ref{prop1}.
  We now have
\begin{align*}
  \dlim\Phi(\X)
  &\subsetref{Lem.~\ref{nylem2}}\Phi(\dlim\X)
  \subseteq \Filt\bigoplus f_*(\dlim\X)
  \subsetref{Lem.~\ref{nylem1}} \Filt\bigoplus\dlim f_*(\X) \\
  &\subseteq \Filt\bigoplus\dlim\ext f_*(\X)
  \subseteq \dlim\ext f_*(\X)
  \subsetref{Rem.~\ref{fviso}}\dlim\ext\Phi(\X) \\
  &\subsetref{Lem.~\ref{nylem2}}\dlim\Phi(\ext\X) 
  \subseteq \dlim\Phi(\X).
\end{align*}
\end{enumerate}
\end{proof}

As mentioned in the introduction we also get $iii)$
by combining results in \cite{holm} and \cite{stovicek13} but for a more restrictive class of abelian categories.

As a special case we get the known results from
\cite{enochsflat} and \cite{enochsproj}.
\begin{lemma}\label{fgprojchar}
  Let $\A$ be an AB5-abelian category, let $Q$ be a left-rooted quiver and let $\X\subseteq\A$ be a set of projective objects.
  Then
  \begin{enumerate}
  \item[i)]   $\Phi(\bigoplus X)=\bigoplus f_*(\X)=\bigoplus\Phi(\X)$
  \item[ii)]   $\Phi(\Add X)=\Add f_*(\X)=\Add\Phi(\X)$
  \end{enumerate}
  If $\A$ has enough projectives, then
    \begin{enumerate}
    \item[iii)]   $\Phi(\Proj\A)=\Proj(\Rep(Q,\A))$
    \end{enumerate}
  If $\A$ is locally finitely presented, generated by $\proj(\A)$ (the finitely generated projective objects) then 
    \begin{enumerate}
    \item[iv)]   $\Phi(\dlim\proj(\A))=\dlim\proj((\Rep(Q,\A)))$
    \end{enumerate}
\end{lemma}
\begin{proof}
  For $i)$ and $ii)$ just notice that any filtration is a sum as all extensions of projective objects are split.
  For $iii)$ and $iv)$ we notice that
  if $\X=Proj(\A)$ (resp. $\X=\proj(\A)$) generate $\A$ then 
  $f_*(\X)\subseteq\Proj(\Rep(Q,\A))$ (resp. $f_*(\X)\subseteq\proj(\Rep(Q,\A)$) generate $\Rep(Q,\A)$.
  Hence $\Add f_*(\X)=\Proj(\Rep(Q,\A).$ (resp. $\add f_*(\X)=\proj(\Rep(Q,\A)$). Now use Theorem~\ref{thma} $ii)$ (resp. v))
\end{proof}
As noted in the introduction, $iii)$ can be seen by using cotorsion pairs as in \cite{holm}.
In the rest of the paper we study the Gorenstein situation.



\section{Gorenstein projective objects}
We will now define the small (i.e. $FP_{2.5}$) Gorenstein projective objects and describe their direct limit closure using Theorem~\ref{thma}.

\begin{defn}
Let $\mathscr{A}$ be an abelian category and $\mathscr{P}$ a class of objects in $\mathscr{A}$.
A complete $\mathscr{P}$-resolution is an exact sequence
with components in $\mathscr{P}$
that stays exact after applying $\textnormal{Hom}(P,-)$
and $\textnormal{Hom}(-,P)$ for any $P\in\mathscr{P}$.

We say that $X$ has a complete $\mathscr{P}$-resolution if it is a syzygy in a complete
$\mathscr{P}$-resolution, i.e. if there exists a complete $\mathscr{P}$-resolution
	\[ \ldots \to P_{1} \to P_0 \to P_{-1} \to \ldots \]
s.t. $X=\textnormal{ker}(P_0\to P_{-1}).$
\end{defn}

We say that $X\in\mathscr{A}$ is Gorenstein projective (resp. Gorenstein injective) if it has a complete $\mathscr{P}$-resolution where
$\mathscr{P}$ is the class of all projective (resp. injective) objects.

We let $\GProj(\mathscr{A})$ (resp. $\GInj(\mathscr{A})$)) denote the Gorenstein projective
(resp. Gorenstein injective) objects  of $\mathscr{A}$ and let
$\Gproj(\mathscr{A}) = \GProj(\mathscr{A}) \cap FP_{2.5}(\mathscr{A})$.

\begin{remark}\label{remarksquare}
Notice that the class $\GProj(\A)$ is closed under extensions see \cite[thm 2.5]{holm04}. Hence so is $\Gproj(\A)$ by Lemma~\ref{lemmadelta}. 
\end{remark}

Dually to the already mentioned characterization of the projective representations, we have a characterization of the injective representations. This was first noted in \cite{enochsinj} and generalized to abelian categories in \cite{holm}.
A similar description is possible for Gorenstein projective and Gorenstein injective objects as proven first for modules over Gorenstein rings in \cite{enochsinj} and then modules over arbitrary rings in \cite[Thm. 3.5.1]{eshragi13}. This proof work in any abelian category with enough projective (resp. injective) objects.
We collect the results here for ease of reference.
\begin{thm}\label{thm8} Let $Q$ be a left-rooted quiver, $\A$ an abelian category with enough projective objects, and $\B$ a category with enough injective objects. Then
\begin{align*}
  \Proj(\Rep(Q,\A))    &= \Phi(\Proj(\A))  \\
  \GProj(\Rep(Q,\A))   &= \Phi(\GProj(\A)) \\
  \Inj(\Rep(Q\op,\B))  &= \Psi(\Inj(\B))   \\
  \GInj(\Rep(Q\op,\B)) &= \Psi(\GInj(\B))
\end{align*}
where for $\Y\subseteq\B$ we define
\[
  \Psi(\Y) =\{ F \in \Rep(Q\op,\B)
    \mid
    \forall v\in Q_0: \psi^F_v \textnormal{ epi and } \ker \psi^F_v \in \X
  \}
\] and
\[ \psi^F_v = F(v) \to \prod_{v\to w} F(w).\]
\end{thm}
As mentioned in the proofs, left-rooted is not needed for the inclusions ($\subseteq$) in the non-Gorenstein cases. We note that $f_v$ preserves Gorenstein projectivity:
\begin{lemma}\label{star2}
Suppose $\mathscr{A}$ satisfies AB4${}^*$ or has enough projective objects or $Q$ is locally path-finite. If $X\in\mathscr{A}$ is Gorenstein projective,
then so is $f_v(X)\in \textnormal{Rep}(Q,\mathscr{A})$.
\end{lemma}

\begin{proof}
Let $P_\bullet$ be a complete projective resolution of $X$.
Then $f_v(P_\bullet)$ is exact and has projective components by \Rref{delta}.

Obviously $\Hom(P, f_v(P_\bullet))$ is exact for any projective $P$, and
$\Hom(f_v(P_\bullet), P)\cong \Hom(P_\bullet, e_v(P))$ is exact if $e_v$ preserves projective objects.

If $\mathscr{A}$ has enough projective objects then $\textnormal{Proj}(\textnormal{Rep}(Q,\mathscr{A})) \subseteq \Phi(\textnormal{Proj}\mathscr{A})$. (\Tref{thm8})
If $\mathscr{A}$ satisfies AB4${}^*$ or $Q$ is locally path-finite, then as in the proof of Lemma~\ref{star1} $e_v$ has an exact right-adjoint (see\cite[3.6]{holm}).
In all cases $e_v$ preserves projective objects.
\end{proof}

Using these and \Tref{thma} we have

\begin{proof}[Proof of Theorem~\ref{thmb}]
  By Theorem \ref{thma} and Remark~\ref{remarksquare} we have
\[
\varinjlim \Phi(\Gproj(\A)) = \Phi(\varinjlim \Gproj(\A)) = \varinjlim \ext f_*(\Gproj(\A)) \]
Now $f_v$ preserves smallness (Lemma \ref{star1} $(i)$) and Gorenstein projectivity (Lemma \ref{star2}), so
\[
\varinjlim \ext f_*(\Gproj(\A)) \subseteq \varinjlim \Gproj(\Rep(Q,\A)).\]
If $Q$ is locally path-finite and target-finite,
Theorem \ref{thm8} and Lemma \ref{star1}$(iii)$ give
\begin{align*}
\Gproj(\Rep(Q,\A))
    = \Phi(\GProj(\A)) \cap FP_{2.5}(\Rep(Q,\A)) 
    \subseteq \Phi(\Gproj(\A))
\end{align*}
so$$\dlim \Gproj(\Rep(Q,\A)) \subseteq\dlim \Phi(\Gproj(\A)).$$
If instead $\varinjlim \Gproj(\A) = \varinjlim \GProj(\A)$ then by Theorem~\ref{thm8} 
\begin{align*}
\varinjlim \Gproj(\Rep(Q,\A))
    & \subseteq \varinjlim \GProj(\Rep(Q,\A)) \\
    & \subseteq \varinjlim \Phi(\GProj(\A)) \\
    & \subseteq \Phi(\varinjlim \GProj(\A)) \\
    & = \Phi(\varinjlim \Gproj(\A)). \qedhere
\end{align*}
\end{proof}

\section{Weakly Gorenstein flat objects}
In this section we will first explain what we mean by an abstract Pontryagin dual.
It mimics the behavior of $\Ab(-,\Q/\Z)$.
We will then define and describe the weakly Gorenstein flat objects and show when they equal $\varinjlim gP$.

Recall that a functor $F:\mathscr{C}\to\mathscr{D}$ \emph{creates exactness} when
$A\to B\to C$ is exact in $\mathscr{C}$ if and only if $FA\to FB\to FC$ is exact in $\mathscr{D}$.
\begin{defn}
A Pontryagin dual is a contravariant adjunction between abelian categories that creates exactness.
I.e. let $\mathscr{C},\mathscr{D}$ be abelian categories.
A \emph{Pontryagin dual} between $\mathscr{C}$ and $\mathscr{D}$ consists of two functors
\[
(-)^+:\mathscr{C}^\textnormal{op}\to \mathscr{D}, \quad
(-)^+: \mathscr{D}^\textnormal{op}\to \mathscr{C}
\]
that both create exactness together with a natural ismorphism
$\mathscr{C}(A,B^+)\cong \mathscr{D}(B,A^+)$.

We call it \emph{$\otimes$-compatible} if there is a \emph{continuous} bifunctor $\otimes: \mathscr{D}\times \mathscr{C}\to\mathscr{K}$
to some abelian category $\mathscr{K}$ s.t.
\[
\mathscr{C}(A,B^+)\cong \mathscr{K}(B\otimes A, E) \cong \mathscr{D}(B,A^+)
\]
for some injective cogenerator $E\in\mathscr{K}$ (i.e. $\mathscr K(-,E)$ creates exactness). Here continuous means that it respects direct limits.
\end{defn}
Note that $\Ab(-,\Q/\Z)\colon\Ab^{\text{op}}\to\Ab$ is a Pontryagin dual compatible with the usual tensor product $\otimes\colon\Ab\times\Ab\to\Ab$ with $E = \Q/\Z$.
\begin{ex}\label{dualex}
${}$
As the following examples shows, (abstract) Pontyagin duals abound.
\begin{enumerate}[1)] 
  \item Let $(\mathscr{C}, [-,-], \otimes, 1)$ be a symmetric monoidal abelian category. Let $E\in \mathscr{C}$ be an injective cogenerator s.t. also $[-,E]$ creates exactness.
  Then $[-,E]$ is a $\otimes$-compatible Pontryagin dual,
  and any $\otimes$-compatible Pontryagin dual is of this form.
  It will thus also satisfy 
$$
  [A,B^+]\cong (B\otimes A)^+ \cong [B, A^+].
$$
  This example includes the motivating example
  $\mathscr{C}=\textnormal{Ab}, E = \Q/\Z$ as well as $\mathscr{C}=\textnormal{Ch}(\textnormal{Ab}), E=\Q/\Z$
  (i.e. $\Q/\Z$ in degree 0 and 0 otherwise).
  \item If $(-)^+:\mathscr{C}^\textnormal{op}\to\mathscr{D}$ is a Pontryagin dual it induces a
  Potryagin dual $\textnormal{Fun}(\A, \mathscr{C})^\textnormal{op} \to \textnormal{Fun}(\A^\textnormal{op},\mathscr{D})$
  for any small category $\A$ by applying $(-)^+$ component-wise.

  If $ (-)^+:\mathscr{C}^\textnormal{op}\to\mathscr{D}$ is compatible with $\otimes: \mathscr{D} \times \mathscr{C}\to \mathscr{K}$,
  then $(-)^+:\textnormal{Fun}(\A, \mathscr{C})^\textnormal{op} \to \textnormal{Fun}(\A^\textnormal{op},\mathscr{D})$
  is compatible with $\otimes\colon \textnormal{Fun}(\A^\textnormal{op},\mathscr{D})\times\textnormal{Fun}(\A, \mathscr{C})\to\mathscr K$ where
  $G\otimes F$ is the coend of
  $$\A^{\text{op}}\times\A\to \mathscr D\times\mathscr C\to\mathscr K$$
  i.e. the coequalizer of the two obvious maps
  $$
  \bigoplus_{a\to b}G(b)\otimes F(a)\rightrightarrows\bigoplus_{a\in\A}G(a)\otimes F(a)$$
  provided the required colimits exists. (see Oberst and R{\"o}hrl \cite{oberst70} or Mac Lane \cite[IX.6]{maclane} for this construction).  

  This includes the case $\Rep(Q,\mathscr{C})$ for any quiver $Q$.
  \item As in 2), any Pontryagin dual $(-)^+:\mathscr{C}^\textnormal{op}\to\mathscr{D}$
  gives a component-wise Pontryagin dual
  $\textnormal{Ch}(\mathscr{C})^\textnormal{op} \to \textnormal{Ch}(\mathscr{D})$
  of chain-complexes. If $(-)^+:\mathscr{C}^\textnormal{op}\to\mathscr{D}$ is compatible with
  $\otimes: \mathscr{D}\otimes \mathscr{C}\to \mathscr{K}$ with injective cogenerator $E\in\mathscr K$,
  then $(-)^+:\textnormal{Ch}(\mathscr{C})^\textnormal{op} \to \textnormal{Ch}(\mathscr{D})$
  is compatible with the total tensor product
  $\textnormal{Ch}(\mathscr{D})\times \Ch(\mathscr{C})\to \Ch(\mathscr{K})$, the injective cogenerator beeing $E$ in degree
  $0$ and $0$ otherwise.

  With $\mathscr{C}=\textnormal{Ab}, (-)^+=[-,\Q/\Z]$ this construction gives the standard one in $\Ch(\textnormal{Ab})$
  as mentioned in 1).
  \item If $\mathscr{C}=\mathscr{D}$ is symmetric monoidal with a $\otimes$-compatible Pontryagin dual as in 1) then the dual of a map $A\otimes X \overset{m}{\to} X$ gives a map
$X^+\otimes A\overset{m^+}{\to} X^+$ via the isomorphisms
\[
\Hom(X^+, (A\otimes X)^+ ) \cong \Hom(X^+, [A, X^+])\cong \Hom(X^+\otimes A, X^+).
\]
One can check that if $A$ is a ring object and $m$ is a left multiplication then $m^+$ is a
right multiplication and we get a Pontryagin dual $(-)^+\colon (A$-Mod$)^\textnormal{op}\to$ Mod-$A$ from the category
of left $A$-modules to the category of right $A$-modules.

This is $\otimes$-compatible with $-\otimes_A-\colon$  $($Mod-$A)\times(A$-Mod$)\to\mathscr C$, where $X\otimes_A Y$ is the coequalizer
of the two obvious maps
\[\begin{tikzcd} X\otimes A\otimes Y \ar[shift left]{r} \ar[shift right]{r} & X\otimes Y\end{tikzcd}.\]
(See Pareigis \cite{pareigis77} for the details of this construction).
This gives the standard Pontryagin dual in $R$-Mod for any ring $R$ (i.e. a ring object in Ab),
and by 3) the standard one in $\Ch(R\textnormal{-Mod})$. It also gives the character module of differential graded $A$-modules (DG-$A$-Mod)
when $A$ is a differential graded algebra, i.e. a ring object in $\Ch(\Ab)$ (see Avramov, Foxby and Halperin \cite{avramov}).
By 2) we also get the one in \cite[Cor 6.7]{enochsflat} for $\Rep(Q,R\textnormal{-Mod})$ for any
ring $R$ and quiver $Q$.
\end{enumerate}
\end{ex}

\begin{defn}\label{defn9}
Let $(-)^+:\mathscr{C}^\textnormal{op}\to \mathscr{D}$ be a Pontryagin dual. We say that
\begin{itemize}
  \item $X\in \mathscr{C}$ is \emph{flat} if $X^+$ is injective in $\mathscr{D}$,
  \item $X\in \mathscr{C}$ is \emph{weakly Gorenstein flat} ($\wGFlat$) if $X^+$ is Gorenstein injective.
  \item $F\in\textnormal{Ch}(\mathscr{C})$ is a \emph{complete flat resolution} if
    $F^+$ is a \emph{complete injective resolution} in $\textnormal{Ch}(\mathscr{D})$,
  \item $X\in\mathscr{C}$ is \emph{Gorenstein flat} ($\GFlat$) if it has a (i.e. is a syzygy in a) complete flat resolution,
\end{itemize}
\end{defn}
Gorenstein flat always implies weakly Gorenstein flat. The other implication requires one to construct a complete flat resolution when the dual has a complete injective resolution. We will look at when this is possible in the next section.

With $\otimes$-compatibility these notions agree with the standard notions.

\begin{prop}
If $(-)^+:\mathscr{C}^\textnormal{op}\to \mathscr{D}$ is $\otimes$-compatible,
then
\begin{enumerate}[1)]
  \item $F\in \mathscr{C}$ is flat if and only if $-\otimes F$ is exact.
  \item $F_\bullet$ is a complete flat resolution if and only if $F_i$ is flat
  for all $i$ and $I\otimes F_\bullet$ is exact for all injective objects $I\in \mathscr{D}$.
\end{enumerate}
\end{prop}
\begin{proof}
1) We have the following equivalences
\begin{align*}
  F \textnormal{ is flat} &
  \Leftrightarrow F^+ \textnormal{ is injective} \\
  & \Leftrightarrow \Hom(-, F^+) \textnormal{ is exact} \\
  & \Leftrightarrow \Hom(-\otimes F, E) \textnormal{ is exact for some injective cogenerator } E \\
  & \Leftrightarrow -\otimes F \textnormal{ is exact}
\end{align*}
2)
Let $F_\bullet\in\Ch(\mathscr C).$
Then $F_i^+$ is injective iff $F_i$ is flat
and $\Hom(I, F^+_\bullet)$ is exact iff $\Hom(I\otimes F_\bullet, E)$
is exact for some injective cogenerator $E$ by $\otimes$-compatibility, see \Eref{dualex} 3). But this happens iff $I\otimes F_\bullet$ is exact.
\end{proof}

The following lemma shows how the classes $\Phi(\X)$ (Definition~\ref{defphi}) and $\Psi(\X)$ (see Theorem~\ref{thm8}) behave with respects to the Pontryagin duals. The proofs are straightforward.
\begin{lemma}\label{lemphipsi}
  Let  $(-)^+\colon\A\to\mathscr B$ be a Pontryagin dual between abelian categories, let $Q$ a quiver, let $\X\subseteq\A$ and $\Y\subseteq\B.$
  Then $$\Phi(\X)^+\subset\Psi(\X^+).$$
    In particular if  $\X=\{X\in\A\mid X^+\in\mathscr Y\}$ then
$$F\in \Phi(\mathscr{X}) \Leftrightarrow F^+\in \Psi(\mathscr{Y}).$$
  If $Q$ is target-finite then
  $$\Psi(\Y)^+\subseteq\Phi(\Y^+).$$

\end{lemma}
\begin{proof}
  For the first assertion we must notice, that   $(\phi_v^F)^+=\psi_v^{F^+}$ for all $F\in\Rep(Q,\A)$ and all $v\in Q_0$.
  For the second, that  $(\psi_v^G)^+=\phi_v^{G^+}$ for all $G\in\Rep(Q\op,\B)$ and all $v\in Q_0$ when $Q$ is target-finite.
  This is because the product in the definition of $\psi_v^G\colon G(v)\to\prod_{v\to w\text{ in }Q\op}G(w)$ is finite when $Q\op$ is source-finite, thus it is a sum and so is the dual. 
\end{proof}
This immediately gives the following:
\begin{proof}[Proof of Theorem~\ref{thmc}]
  \begin{align*}
    F\in\Flat(\Rep(Q,\A))&\overset{Def.~\ref{defn9}}\iff F^+\in\Inj(\Rep(Q^{\text{op}},\B)) \\
  &\overset{Thm.~\ref{thm8}}\iff F^+\in\Psi(\Inj(\B)) \\
  &\overset{Lem.~\ref{lemphipsi}}\iff F\in\Phi(\Flat(\A))\end{align*}
The same proof works in the Gorenstein situation.
\end{proof}
\begin{remark}
This gives a straightforward proof of \cite[Thm 3.7]{enochsflat} using the
characterization of the injective representations from \cite{enochsinj}.
\end{remark}
Combining this with Theorem \ref{thmb} we get:
\begin{cor}\label{newprop20}
Let $(-)^+:\A^{\text{op}}\to \mathscr{B}$ be a Pontryagin dual, let $Q$ be a left-rooted quiver and assume
\begin{itemize}
\item $\A$ has enough projective objects
\item $\mathscr{B}$ has enough injective objects
\item $Q$ is target-finite and locally path-finite, or $\varinjlim \Gproj(\mathscr{A}) = \varinjlim \GProj(\mathscr{A}).$
\end{itemize}
If $\varinjlim \Gproj = \wGFlat$ in $\A$ then the same is true in $\Rep(Q,\A)$.
\end{cor}

\section{Gorenstein flat objects}
We will now find conditions on the category $\A$ and the quiver $Q$ s.t.
\[ \wGFlat(\Rep(Q, \A)) = \GFlat(\Rep (Q,\A)). \]
Firstly we have the following known result:
\begin{prop}\cite[Prop. 3.6]{holm04}\label{gorensteinflat1}
Let $R$ be a right coherent ring.
Then $\wGFlat(\Rmod) = \GFlat(\Rmod)$
\end{prop}

Looking more closely at the proof (see Christensen \cite[Thm. 6.4.2]{christensen00}) we arrive at \Lref{gorensteinflat3}. We include a full proof, as our notions of flatness differ.
\begin{lemma}\label{gorensteinflat2}
Let $\A$ be an abelian category.
If $0\to X'\to J\to X \to 0$ is exact
and $J$ is injective (or just Gorenstein injective), $X$ is Gorenstein injective and $\Ext^1(I, X')=0$
for all injective  $I\in \mathscr{A}$.
Then $X'$ is Gorenstein injective.
\end{lemma}
\begin{proof}
This is the dual of \cite[2.11]{holm04}. The proof is for modules but works
in any abelian category.
\end{proof}

Now recall that a class $\X\subseteq\mathscr C$ is \emph{preenveloping} if for every $M\in\mathscr C$ there is a map $\phi\colon M\to X$ called the \emph{preenvelope} to some
$X\in\X$ s.t. every map from $M$ to an object in $\X$ factors through $\phi.$ It is monic whenever there exists some monomorphism from $M$ to an object of $\X.$
\begin{lemma}\label{gorensteinflat3}
Let $(-)^+:\mathscr{C}^\textnormal{op} \to \mathscr{D}$ be a Pontryagin dual and assume
\begin{itemize}
  \item[(1)] $\Inj(\D)^+\subseteq\Flat(\C)$
  \item[(2')] $\textnormal{Flat}(\mathscr{C})$ is preenveloping.
  \item[(3')] $\mathscr{C}$ has enough flat objects.
\end{itemize}
Then any weakly Gorenstein flat object of $\mathscr{C}$ is Gorenstein flat.
\end{lemma}
\begin{proof}
Let $X$ be weakly Gorenstein flat, i.e. $X^+$ is Gorenstein injective.
Our goal is to construct a complete flat resolution for $X$.
The left part of such a resolution is easy when $\mathscr{C}$ has enough flat objects. As $X^+$
is Gorenstein injective it has an injective resolution $I_\bullet$
s.t. $\Hom(J,I_\bullet)$ is
exact for any injective $J$.
But then this holds for any injective resolution of $X^+$.
In particular $F^+_\bullet$, where $F_\bullet$ is a flat
(left-) resolution of $X$ which exists when $\mathscr{C}$ has enough flats.

For the right part we construct the resolution one piece at a time by constructing for any
weakly Gorenstein flat $X\in\C$ a short exact sequence
$0\to X\to F\to X'\to 0$ where $F$ is flat s.t. $\Ext^1(I, X'^+)=0$
for any injective $I\in\mathscr D$. Then $X'^+$ is Gorenstein injective by Proposition~\ref{gorensteinflat1} and this process can be continued to give a
flat (right-) resolution $F_\bullet$
of $X$ s.t. $\Hom(I, F_\bullet^+)$ is exact for any injective $I$.

So let again $X\in\C$ be weakly Gorenstein flat,
and let $\varphi: X\to F$ be a flat preenvelope.
We first show that $\phi$ is monic by showing that there exists some monomorphism from $X$ to a flat object.
Since $X^+$ is Gorenstein injective there exist an epimorphism $E\to X^+$
from some injective $E\in\D$. But then
$X\to X^{++} \to E^+$ is monic, since  $(-)^+$ creates exactness and $X^{+++}\to X^+$ is split epi by the unit-counit relation. Thus $\phi$ is monic since $E^+$ is flat by (1). We thus have a short exact sequence
$0\to X\overset{\phi}{\to} F\to X'\to 0$ inducing for any injective $I\in\mathscr D$ a long exact sequence
\[
0\to \Hom(I,X'^+)\to \Hom(I,F^+)\overset{\varphi^*}{\to} \Hom(I, X^+)
    \to \Ext^1(I, X'^+) \to \Ext^1(I, F^+).
\]
Now $\varphi^*$ is epi as $\varphi_*: \Hom(F,I^+)\to \Hom (X, I^+)$
is epi because $I^+$ is flat and $\varphi$ is a flat preenvelope.
Since $\Ext^1(I,F^+)$ is 0 because $F^+$ is injective we must have
$\Ext^1(I, X'^+)=0$.
\end{proof}

We notice that $\A=\Rmod$ satisfies these conditions when $R$ is right coherent.
((1) is Xu \cite[Lem. 3.1.4]{xu}) and (2') is \cite[Prop. 6.5.1]{relhomalg}).\todo{tjek referencer}
Our task is thus to find conditions on $Q$ s.t. the conditions from
Lemma \ref{gorensteinflat3} lift from $\A$ to $\Rep(Q,\A)$.
Lifting the condition that the flat objects are preenveloping is not obvious.
But being closed under products is sometimes enough as the next lemma shows.
We will reuse standard results on purity and therefore assume our Pontryagin Dual is $\otimes$-compatible and $\A$ to be generated by $\proj(\A).$

\begin{lemma}\label{gorensteinflat4}
Let $\A$ be a locally finitely presented abelian category with a Pontryagin dual and assume that
\begin{enumerate}[(1)]
  \item[(2)] The flat objects are closed under products
  \item[(3)] $\A$ is generated by $\proj(\A)$
  \item[(4)] The Pontryagin dual is $\otimes$-compatible
\end{enumerate}
Then the flat objects are preenveloping
\end{lemma}
\begin{proof}
Let $X\in \A$. The idea (as in \cite[Prop. 6.2.1]{relhomalg}) is to find a set of flat objects $\mathscr{S}$ s.t. 
every map $X\to Y$ with $Y$ flat factors as $X\to S\hookrightarrow Y$ 
with $S\in \mathscr{S}$.
Then we can construct a flat preenvelope as 
\[
X\to \prod_{\begin{matrix}
              S\in \mathscr{S} \\
              \varphi: X\to S
            \end{matrix}} 
            S_\varphi,
\]
with $S_\varphi = S$ because the flat objects are closed under products by (2). 

As in in the proof of \cite[Lemma 5.3.12]{relhomalg} there is a set of objects
 $\mathscr{S}\subseteq \A$ s.t. every map $X\to Y$ 
to some $Y\in\A$ factors as $X\to S\hookrightarrow Y$ 
for some $S\in \mathscr{S}$ with the property that, 
given a commutative square
\[
\begin{tikzcd}
L_0 
\ar[hookrightarrow]{r}
\ar{d}
    & 
L_1
\ar[dashed]{dl}
\ar{d}
    \\
S
\ar[hookrightarrow]{r}
    &
Y
\end{tikzcd}
\]
with $L_0$ finitely generated and $L_1$ finitely presented there is a lift $L_1\to S$ s.t. the left triangle commutes.
The proof is for modules and bounds size of $\mathscr S$ by some cardinality. If we are not interested in the cardinality, the proof works in any well-powered category, i.e. a category where there is only a set of subobjects of any given object.
As in Ad\'amek and and Rosick\'y \cite{adamek94} \todo{find reference} any Grothendieck category is well-powered.
We are left with proving that if $Y$ is flat, so is $S$,
i.e. if $Y^+$ is injective, so is $S^+$. 
Now Jensen and Lenzing \cite[Prop. 7.16]{jensen89} shows (using (3)) 
that the above lifting property implies (in fact is equivalent to) that 
$S\hookrightarrow Y$ is a direct limit of split monomorphisms. 
\cite[Thm 6.4]{jensen89} then shows (using (4))
that this implies, that $Y^+\to S^+$ 
is split epi. (Equivalence of these statements uses that 
the generators in $\Rmod$ are \emph{dualizable}).
Thus if $Y^+$ is injective, so is $S^+$.
\end{proof}

\begin{lemma}\label{lem:1-5}
  Let $(-)^+\colon \A\op\to\B$ be a Pontryagin dual where $\A$ is AB4${}^*$ and $\B$ has enough injective objects. Let $Q$ be a left-rooted and target-finite quiver. If $\A$ satisfies $(1)$-$(4)$ (from Lemma \ref{gorensteinflat3} and \ref{gorensteinflat4}) then so does $\Rep(Q,\A)$.
\end{lemma}
\begin{proof}
For (1) we have 
\begin{align*}
\Inj(\Rep(Q^{\text{op}},\B))^{++}
    &\subsetref{Thm. \ref{thm8}} \Psi(\Inj(\B))^{++}
    \subsetref{Lem. \ref{lemphipsi}} \Psi(\Inj(\B)^{++}) \\    &\subseteq \Psi(\Inj(\B))
    \subsetref{Thm. \ref{thm8}} \Inj(\Rep(Q,\B))
\end{align*}
since $\B$ has enough injective objects and $Q$ is left-rooted and target-finite.
For (2) we have 
\begin{align*}
\textstyle\textstyle\prod \Flat(\Rep(Q,\A))
    &\subsetref{\text{Thm. }\ref{thmc}}\textstyle\prod \Phi(\Flat \A)
    \subsetref{Lem. \ref{nylem2}} \Phi(\textstyle\prod \Flat \A) \\    &\subseteq \Phi(\Flat \A)
    \subsetref{\text{Thm. }\ref{thmc}} \Flat(\Rep(Q,\A)).
\end{align*}
since $\A$ is AB4${}^*$ and $\B$ has enough injective objects and $Q$ is left-rooted and target-finite.
(3) and (4) lifts without conditions on $\A$ and $Q.$
For (3), if $\A$ is generated by a set $\mathscr{X}$ of 
finitely generated projective objects then $f_*(\mathscr{X})$ is a 
generating set of finitely generated projective objects by Lem. \ref{star1}$(i)$ and Remark \ref{delta}. (4) is lifted in Example \ref{dualex}.
\end{proof}

Notice that $(3)-(4)$ holds for $\A=\Rmod$ over any ring $R$,
and $(2)$ is equivalent to $R$ being right coherent \cite[Prop. 3.2.24]{relhomalg}.

\begin{prop}\label{propd}
  Let $\A$ be a locally finitely presented abelian  $AB4^*$-category. Let $\B$ be an abelian category with enough injective objects, let $(-)^+\colon \A\op\to\B$ be a $\otimes$-compatible Pontryagin dual.
  If
  \begin{itemize}
  \item $\A$ is generated by $\proj(\A)$
  \item $\Flat(\A)$ is closed under products
  \item $\Inj(\B)^+\subseteq\Flat(\A)$
  \end{itemize}
  then $\Flat(\A)$ is preenveloping and  $\wGFlat(\A)=\GFlat(\A).$
  Assume further that $Q$ is a left-rooted and target-finite quiver. Then $\Flat(\Rep(Q,\A))$ is preenveloping and
  $$\wGFlat(\Rep(Q,\A))=\GFlat(\Rep(Q,\A)).$$
\end{prop}
\begin{proof}
  This follows from Lemma \ref{gorensteinflat3} and \ref{gorensteinflat4} and \ref{lem:1-5}.
We also need (3') to hold and we could lift this directly
by noting that $f_v$ respects flatness, but it also follows from (3).
\end{proof}
We can now prove
\begin{proof}[Proof of Theorem~\ref{thmd}]
Use Proposition~\ref{propd} and the remark above it.
\end{proof}

\begin{remark}\label{rem:wgf=gf}
  In \cite[Lem 6.9 and proof of Thm. 6.11]{enochsinj} it is proved that $$\wGFlat(\Rep(Q,\Rmod))=\GFlat((\Rep(Q,\Rmod))$$
  when $R$ is Iwanaga-Gorenstein and $Q$ is only required to be left-rooted.
  Theorem~\ref{thmd} thus weakens the condition of $R$ but must then strengthen the conditions on $Q.$
\end{remark}

\begin{cor}\label{limgp=gf}
  Let $Q$ be a left-rooted quiver and let $\A$ be as in Proposition~\ref{propd}.
  If $\A=\Rmod$ for some Iwanaga-Gorenstein ring $R$ or
  \begin{itemize}
  \item  $\varinjlim \Gproj(\A) = \GFlat(\A)$ and
  \item $Q$ is target-finite and locally path-finite
  \end{itemize}
Then
$$\varinjlim\Gproj(\Rep(Q,\Rmod) = \GFlat(\Rep(Q,\Rmod)=\Phi(\GFlat(\Rmod)).$$
\end{cor}
\begin{proof}
  Apply Corollary~\ref{newprop20} and Proposition~\ref{propd} (or Remark~\ref{rem:wgf=gf} for the Gorenstein case) to get $\varinjlim\Gproj = \wGFlat = \GFlat$ in $\Rep(Q,\A)$. The last equality then follows from Theorem~\ref{thmc}.
\end{proof}

\section*{Acknowledgements}
I want to thank my advisor, Henrik Holm, for suggesting the topic and for numerous discussions. I would also like to thank Jan \v S\v toví\v cek for providing the proof of \Pref{prop1} on a napkin when I caught him on his way to lunch. Finally I must thank student Asbjørn Bækgaard Lauritsen for typing the manuscript after I became electromagnetic hypersensitive.
\bibliographystyle{amsplain}
\bibliography{bibliography}{}

\providecommand{\bysame}{\leavevmode\hbox to3em{\hrulefill}\thinspace}
\providecommand{\MR}{\relax\ifhmode\unskip\space\fi MR }
\providecommand{\MRhref}[2]{%
  \href{http://www.ams.org/mathscinet-getitem?mr=#1}{#2}
}
\providecommand{\href}[2]{#2}
\begin{thebibliography}{10}

\bibitem{adamek94}
{Ji\v r\'\i} Ad\'amek and {Ji\v r\'\i} Rosick{\'y}, \emph{Locally presentable
  and accessible categories}, London Mathematical Society Lecture Note Series,
  vol. 189, Cambridge University Press, Cambridge, 1994. \MR{1294136}

\bibitem{avramov}
L.~Avramov, H.~Foxby, and S.~Halperin, \emph{Differential graded homological
  algebra}, 1994–2014, preprint.

\bibitem{krause08}
A.~Beligiannis and H.~Krause, \emph{Thick subcategories and virtually
  {G}orenstein algebras}, Illinois J. Math. \textbf{52} (2008), no.~2,
  551--562. \MR{2524651}

\bibitem{breit70}
S.~Breitsprecher, \emph{Lokal endlich pr\"asentierbare
  {G}rothendieck-{K}ategorien}, Mitt. Math. Sem. Giessen Heft \textbf{85}
  (1970), 1--25. \MR{0262330}

\bibitem{christensen00}
L.~W. Christensen, \emph{Gorenstein dimensions}, Lecture Notes in Mathematics,
  vol. 1747, Springer-Verlag, Berlin, 2000. \MR{1799866}

\bibitem{boevey94}
W.~Crawley-Boevey, \emph{Locally finitely presented additive categories},
  Communications in Algebra \textbf{22} (1994), no.~5, 1641--1674.

\bibitem{enochsproj}
E.~E. Enochs and S.~Estrada, \emph{Projective representations of quivers},
  Comm. Algebra \textbf{33} (2005), no.~10, 3467--3478. \MR{2175445}

\bibitem{enochsinj}
E.~E. Enochs, S.~Estrada, and J.~R. Garc{\'\i}a~Rozas, \emph{Injective
  representations of infinite quivers. {A}pplications}, Canad. J. Math.
  \textbf{61} (2009), no.~2, 315--335. \MR{2504018}

\bibitem{relhomalg}
E.~E. Enochs and O.~M.~G. Jenda, \emph{Relative homological algebra}, De
  Gruyter Expositions in Mathematics, vol.~30, Walter de Gruyter \& Co.,
  Berlin, 2000. \MR{1753146}

\bibitem{enochsflat}
E.~E. Enochs, L.~Oyonarte, and B.~Torrecillas, \emph{Flat covers and flat
  representations of quivers}, Communications in Algebra \textbf{32} (2004),
  no.~4, 1319--1338.

\bibitem{eshragi13}
H.~Eshraghi, R.~Hafezi, and Sh. Salarian, \emph{Total acyclicity for complexes
  of representations of quivers}, Communications in Algebra \textbf{41} (2013),
  no.~12, 4425--4441.

\bibitem{holm04}
H.~Holm, \emph{Gorenstein homological dimensions}, J. Pure Appl. Algebra
  \textbf{189} (2004), no.~1-3, 167--193. \MR{2038564}

\bibitem{holm11}
H.~Holm and P.~J\o~rgensen, \emph{Rings without a {G}orenstein analogue of the
  {G}ovorov-{L}azard theorem}, Q. J. Math. \textbf{62} (2011), no.~4, 977--988.
  \MR{2853225}

\bibitem{holm}
H.~Holm and P.~{J\o rgensen}, \emph{Cotorsion pairs in categories of quiver
  representations}, to appear in Kyoto J. Math., 23 pp., April 2016, Preprint.

\bibitem{jensen89}
C.~U. Jensen and H.~Lenzing, \emph{Model-theoretic algebra with particular
  emphasis on fields, rings, modules}, Algebra, Logic and Applications, vol.~2,
  Gordon and Breach Science Publishers, New York, 1989. \MR{1057608}

\bibitem{maclane}
S.~Mac~Lane, \emph{Categories for the working mathematician}, second ed.,
  Graduate Texts in Mathematics, vol.~5, Springer-Verlag, New York, 1998.
  \MR{1712872}

\bibitem{oberst70}
U.~Oberst and H.~R{\"o}hrl, \emph{Flat and coherent functors}, J. Algebra
  \textbf{14} (1970), 91--105.

\bibitem{pareigis77}
B.~Pareigis, \emph{Non-additive ring and module theory. {I}. {G}eneral theory
  of monoids}, Publ. Math. Debrecen \textbf{24} (1977), no.~1-2, 189--204.
  \MR{0450361}

\bibitem{stenstrom70}
B.~Stenstr\"om, \emph{Coherent rings and {$F\,P$}-injective modules}, J. London
  Math. Soc. (2) \textbf{2} (1970), 323--329. \MR{0258888}

\bibitem{stovicek13}
J.~{\v{S}}{\v{t}}ov{\'{\i}}{\v{c}}ek, \emph{Deconstructibility and the {H}ill
  lemma in {G}rothendieck categories}, Forum Math. \textbf{25} (2013), no.~1,
  193--219.

\bibitem{xu}
J.~Xu, \emph{{Flat covers of modules.}}, vol. 1634, Berlin: Springer, 1996
  (English).

\end{thebibliography}

\end{document}